\documentclass[a4paper,reqno]{amsart}
\usepackage[czech,english]{babel}
\usepackage{amscd,amssymb}
\usepackage{dsfont}
\usepackage{url}
\usepackage[all]{xy}
\usepackage{tikz}
\usepackage{tikz-cd}
\usepackage{hyperref}
\usepackage{enumitem}

\newcommand{\rest}{\upharpoonright}
\newcommand{\mathcolon}{\colon\,} %% Hovey uses for maps, like f: A -> B
\newcommand{\rmod}[1]{\mbox{\rm{Mod}--}{#1}}

\DeclareMathOperator{\lse}{G}
\DeclareMathOperator{\lsm}{S}
\newtheorem{example}{Example}
\title{Periodic modules and acyclic complexes}
\author[S. Bazzoni]{Silvana Bazzoni}
\address[Silvana Bazzoni]{%
Dipartimento di Matematica \\
Universit\`a di Padova \\
    Via Trieste 63, 35121 Padova (Italy)}
\thanks{The first named autor is partially supported by grants BIRD163492 and DOR1690814 of
  Padova University}
\email{bazzoni@math.unipd.it}
\author[M. Cort\'es Izurdiaga ]{Manuel Cort\'es Izurdiaga}
\address[Manuel Cort\'es Izurdiaga]{%
Departamento de Did\'actica de las Matem\'aticas, Did\'actica de las
Ciencias Sociales y de las Ciencias Experimentales \\
Universidad de M\'alaga \\
 29071 M\'alaga (Spain)}
\thanks{The second named author is partially supported by grants
  MTM2014-54439 and MTM2016-77445-P of Ministerio de Econom\'{\i}a,
  Industria y Competitividad}
\email{mizurdiaga@uma.es}

\author[S. Estrada ]{Sergio Estrada.}
\address[Sergio Estrada]{%
Departamento de Matem\'aticas \\
Universidad de Murcia \\
 Campus de Espinardo, 30100 Murcia (Spain)}
\thanks{The third named author is partially supported by grants
  18934/JLI/13 of Fundaci\'on S\'eneca-Agencia de Ciencia y
  Tecnolog\'{\i}a de la Regi\'on de Murcia in the framework of III
  PCTRM 2011-2014, MTM2016-77445-P of Ministerio
  de Econom\'{\i}a, Industria y Competitividad and FEDER funds}
\email{sestrada@um.es}

\newcommand{\dd}{\colon}

\newcommand{\bbZ}{\mathbb{Z}}

\newcommand{\Hom}{\operatorname{Hom}}

\newcommand{\Ext}{\operatorname{Ext}}

\newcommand{\Ker}{\operatorname{Ker}}

\newcommand{\Img}{\operatorname{Im}}

\newcommand{\A}{\mathcal{A}}
\newcommand{\B}{\mathcal{B}}
\newcommand{\C}{\mathcal{C}}
\newcommand{\D}{\mathcal{D}}
\newcommand{\E}{\mathcal{E}}

\newcommand{\G}{\mathcal{G}}

 \newcommand{\Ch}{\mathrm{Ch}}
\newcommand{\tilclass}[1]{\widetilde{#1}}

\newcommand{\Modr}[1]{\mathrm{Mod}\textrm{-}{#1}}

\newcommand{\Flat}{\mathrm{Flat}}
\newcommand{\Cot}{\mathrm{Cot}}
\newcommand{\Inj}{\mathrm{Inj}}
\newcommand{\PInj}{\mathrm{PInj}}

\newcommand{\Proj}{\mathrm{Proj}}
\newcommand{\PProj}{\mathrm{PProj}}
\newcommand{\ModR}{\mathrm{Mod}\textrm{-}R}

\newcommand{\add}{\mathrm{add}}
\newcommand{\Add}{\mathrm{Add}}

\DeclareMathOperator{\ex}{ex}
\DeclareMathOperator{\dw}{dw}
\DeclareMathOperator{\dg}{dg}
\DeclareMathOperator{\Z}{Z}
\DeclareMathOperator{\Abs}{Abs}
\DeclareMathOperator{\FPD}{FPD}

\DeclareMathOperator{\Min}{Min}
\DeclareMathOperator{\FID}{FID}

\theoremstyle{plain}
\newtheorem{thm}{Theorem}[section]
\newtheorem{lem}[thm]{Lemma}
\newtheorem{prop}[thm]{Proposition}
\newtheorem{cor}[thm]{Corollary}
\newtheorem{ques}[thm]{Question}

\theoremstyle{definition}
\newtheorem{defn}[thm]{Definition}

\theoremstyle{remark}
\newtheorem{rem}[thm]{Remark}

\begin{document}


\begin{abstract}
  We study the behaviour of modules $M$ that fit into a short exact
  sequence $0\to M\to C\to M\to 0$, where $C$ belongs to a class of
  modules $\C$, the so-called $\C$-periodic modules. We find a rather
  general framework to improve and generalize some well-known results
  of Benson and Goodearl and Simson. In the second part we will
  combine techniques of hereditary cotorsion pairs and
  presentation of direct limits, to conclude, among other
  applications, that if $M$ is any module and $C$ is cotorsion, then
  $M$ will be also cotorsion. This will lead to some meaningful
  consequences in the category $\Ch(R)$ of unbounded chain complexes
  and in Gorenstein homological algebra. For example we show that
  every acyclic complex of cotorsion modules has cotorsion cycles, and
  more generally, every map $F\to C$ where $C$ is a complex of
  cotorsion modules and $F$ is an acyclic complex of flat cycles, is
  null-homotopic. In other words, every complex of cotorsion modules
  is dg-cotorsion.
\end{abstract}

\maketitle

% -----------------------------------------------------------------------------
% The table of contents
\setcounter{tocdepth}{1}
% \tableofcontents

% ------------------------------------------------------------------------------
% The document body

\section*{Introduction}

Throughout this paper $R$ is an associative ring with identity and all
modules will be right $R$-modules.
 
The goal of this work is the study of periodic and pure periodic
modules with respect to an arbitrary class of modules $\C$. More
precisely, one of the main objectives we pursue is to know when
$\C$-periodic modules (resp. pure $\C$-periodic modules) are trivial,
where an $R$-module $M$ is called \emph{$\C$-periodic}
(resp. \emph{pure $\C$-periodic}) if it fits into an exact sequence
(resp. into a pure exact sequence) of the form
$0\to M\to C\to M\to 0$, with $C\in \C$, and it is called
\emph{trivial} if it belongs to $\C$.  The origin of this problem
comes from the celebrated result by Benson and Goodearl \cite[Theorem
2.5]{BG} in which they show that each flat $\Proj$-periodic module is
trivial (here $\Proj$ denotes the class of all projective modules). It
is then easy to observe that Benson and Goodearl statement can be
reformulated to saying that each pure $\Proj$-periodic module is
trivial. This is because $M$ is always flat in each pure short exact
sequence of the form $0\to M\to P\to M\to 0$, with $P$ projective.

This module-theoretic property has a remarkable consequence at the
level of chain complexes of modules: every acyclic complex of
projective modules with flat cycles is contractible. This connection
between flat $\Proj$-periodic modules and acyclic complexes of
projective modules was firstly observed by Christensen and Holm
\cite{CH} and entitled them to find the module-theoretic proof
conjectured by Neeman in \cite[Remark 2.15]{Nee08} to get the
aforementioned result on acyclic complexes of projective modules with
flat cycles (Neeman already gives a proof of this fact by using
homotopy techniques). In 2002 Simson \cite{Sim02} realized that Benson
and Goodearl's theorem can be established in the pure setting of a
finitely accesible Grothendieck category, by showing that if
$M$ fits into a pure exact sequence $0\to M\to T\to M\to 0$, with $T$
pure projective (i.e. a direct summand of a direct sum of finitely
presented objects), then $M$ itself is pure projective. In other
words, every pure $\PProj$-periodic module is trivial ($\PProj$ is the
class of all pure projective modules).

We devote the first part of this paper to show that both Benson and
Goodearl and Simson results are encoded under the following rather
more general statement (see Theorem \ref{t:SplittingSequence},
Corollary \ref{c:SplittingSequence}(1) and (2)):

\begin{thm}\label{Intr.Thm 1}
  Any short exact sequence
  $0\to M\hookrightarrow G\stackrel{g}{\to} M\to 0$ in which $g$ is
  locally split and $G$ is a direct sum of countably generated modules
  is split.
\end{thm}

Aside from the preceding particular instances of this Theorem, we also
get (Corollary \ref{c:SplittingSequence}) that each pure
$\Add(\mathcal P)$-periodic module is trivial (here $\mathcal P$
denotes a class of finitely presented modules and $\Add(\mathcal P)$
is the class of direct summands of direct sums of modules in
$\mathcal P$).

Now, if we think on a flat module $F$ as such that every short exact
sequence $0\to M\to N\to F\to 0$ is pure, we immediately realize that
the dual notion of flat module is that of absolutely pure
(=FP-injective) module (i.e. a module $E$ such that each exact
sequence of modules $0\to E\to L\to T\to 0$ is pure). Thus it seems
natural to wonder whether or not the dual version of Benson and
Goodearl's theorem holds. Namely, is every
$\Inj$-periodic absolutely pure module injective? ($\Inj$ is the class of all
injective modules). Or even more generally, is every pure
$\PInj$-periodic module trivial? (here $\PInj$ stands for the category
of pure injective modules, i.e. modules $L$ such that each pure short
exact sequence $0\to L\to A\to B\to 0$ splits).  Recently, the answer
to this question has been positively settled by
{\v{S}}{\v{t}}ov{\'{\i}}{\v{c}}ek \cite[Corollary 5.5]{Stopurity} by
using complete cotorsion pairs in the category of unbounded chain
complexes.

Thus, at this point, we know that every pure $\PProj$-periodic module
is trivial and that every pure $\PInj$-periodic module is also
trivial. For the global case we know that flat $\Proj$-periodic
modules are trivial and that $\Inj$-periodic absolutely pure modules are also
trivial. But there are rings for which there exist non trivial
$\mathcal C$-periodic modules for the classes $\Proj$, $\Inj$ and
$\Flat$. These non trivial $\mathcal C$-periodic modules can be
constructed when $\mathcal C$ is a proper generating and cogenerating
class, see Corollary \ref{p:ConstructionPeriodic}.

However when considering the class $\Cot$ of cotorsion modules
(i.e. the right Ext-orthogonal with respect to the flat modules) one
confronts the major problem that this class is almost never generating
(except in the trivial case in which the ring is right perfect, when
all modules are cotorsion). This
is due to a result of Guil Asensio and Herzog \cite[Corollary 10]{GH}
for which we include here a short proof (Theorem
\ref{t:sigmacot}). This suggests that $\Cot$-periodic modules might be
trivial, and indeed one of the main applications of the second part of
the paper is to show that this is the case.

Hence we devote the second part of this paper to address, among
others, this question. We will work directly with hereditary
cotorsion pairs in the module category, rather than in the category of
complexes. So we get a slightly more direct proof of the dual of
Benson and Goodearl Theorem. Also our methods seem to be more far
reaching as they allow to prove, as we announced before, that each
$\Cot$-periodic module is trivial and also allow to get significant
consequences in Gorenstein homological algebra. We state below the
main applications based on the general theorem (Theorem
\ref{T:direct-limits}). The reader can find the proof of these
applications in Proposition \ref{P:Fp-inj} and Corollaries
\ref{c:GorInj} and \ref{c:pInjperiodic.trivial}.

\begin{thm}\label{Intr.Thm 2}
  The following hold true:
  \begin{enumerate}
  \item Every $\Inj$-periodic absolutely pure module is trivial.
  \item Let $(\A, \B)$ be a hereditary cotorsion pair in
    $\Modr R$. Assume that $\A$ is closed under pure epimorphic
    images. Then every $\B$-periodic module is cotorsion. In
    particular, every $\Cot$-periodic module is trivial (and so every
    $\Inj$-periodic module is cotorsion).
  \item Assume that each finitely generated right ideal has finite
    flat dimension. Then each Gorenstein injective module is
    injective. In particular if a ring has finite weak global
    dimension, every Gorenstein injective module is injective.
  \item Assume that each finitely generated right ideal has finite
    injective dimension. Then each Gorenstein injective module is
    injective.
  \item Every pure $\PInj$-periodic module is trivial.
  \end{enumerate}
\end{thm}

The third and fourth statements are applications of Theorem
\ref{T:direct-limits} to the realm of Gorenstein homological
algebra. The first statement shows the dual version of Benson and
Goodearl Theorem. The second statement has very interesting, and
perhaps surprising, consequences for chain complexes of modules as we
will indicate now. The last part of the paper (Section
\ref{S:acyclic}) is therefore devoted to infer these and other
applications of $\C$-periodic modules for chain complexes.

We have already mentioned the relationship observed by Holm and
Christensen between flat $\Proj$-periodic modules and acyclic chain
complexes of projective modules with flat cycles. In a recent paper
Estrada, Fu and Iacob \cite{EFI} show that Christensen's and Holm's
argument can be easily extended to provide a nice correspondence
between $\C$-periodic modules in a class $\D$ of modules, and acyclic
complexes of modules in $\C$ with cycles in $\D$ (see Proposition
\ref{P:cycles} for a precise formulation of the statement). This
bridge between periodic modules and acyclic complexes, is the key to
find applications of our results in the category $\Ch(R)$ of unbounded
chain complexes.  The first application is a consequence of Theorem
\ref{Intr.Thm 2}(2). We recall that a chain complex $C$ is called
\emph{dg-cotorsion} if each $C_n$ is a cotorsion module and each map
$f\mathcolon F\to C$ from an acyclic complex $F$ of flat modules with
flat cycles, is null-homotopic.
\begin{thm}
  Every acyclic complex of cotorsion modules has cotorsion cycles. As
  a consequence, every complex of cotorsion modules is $dg$-cotorsion.
\end{thm}
The proof of this Theorem can be found in theorems \ref{T:cot-Fp} and
\ref{T:dgcotdwcot}.  As a particular instance, we get that every
acyclic complex of injectives has cotorsion cycles. This was already
proved by {\v{S}}{\v{t}}ov{\'{\i}}{\v{c}}ek \cite[Corollary
5.9]{Stopurity}.

The second application (Corollary \ref{C:Stov.Cot.Pair}) also recovers
a result of {\v{S}}{\v{t}}ov{\'{\i}}{\v{c}}ek (\cite[Theorem
5.5]{Stopurity}):
\begin{cor}
  Let $\G$ be a finitely accessible additive category, and let
  $\widetilde{\mathrm{PurAc}}$ be the class of pure acyclic complexes
  in $\Ch(\G)$ (i.e. concatenation of pure short exact sequences in
  $\mathcal G$) and $dw\PInj$ be the class of chain complexes with
  pure injective components. The pair of classes
  $(\widetilde{\mathrm{PurAc}},dw\PInj)$ is a complete hereditary
  cotorsion pair in $\Ch(\G)_{dw\textrm{-}pur}$ (this is the category
  $\Ch(\G)$ of unbounded chain complexes with the degreewise pure
  exact structure induced from $\G$).
\end{cor}
% In \cite[Theorem 8.6]{Nee08} Neeman proves that the pair
% $(dw\Proj,\tilclass{\Flat})$ is a complete cotorsion pair in
% $\Ch(\Flat)$. Our last application gives the
% dual %version of this theorem for coherent rings:
% \begin{cor}
%   Let $R$ be a right coherent ring. The pair of classes
%   $(\tilclass{\FpInj},dw\Inj)$ is a complete cotorsion pair in
%   $\Ch(\FpInj)$.
% \end{cor}

\section{Preliminaries}

Given an ordinal $\kappa$, a family of sets
$\{A_\alpha\mid \alpha < \kappa\}$ is called \emph{continuous} if
$\bigcup_{\alpha < \lambda}A_\alpha = A_\lambda$ for each limit
ordinal $\lambda$ smaller than $\kappa$. Given a map
$f\mathcolon A \rightarrow B$ and $C$ a subset of $A$, we shall denote
by $f \rest C$ the restriction of $f$ to $C$. The cardinality of a set
$A$ will be denoted by $|A|$; $\omega$ will be the first infinite
ordinal number.

Let $\mathcal X$ be a class of modules containing all projective
modules and $n$ a natural number. We shall denote by $\mathcal X_n$ the class of all modules
with $\mathcal X$-projective dimension less than or equal to $n$ (a
module $M$ is said to have $\mathcal X$-projective dimension less than
or equal to $n$ if there exists a projective resolution of $M$ such
that its $(n-1)st$ syzygy belongs to $\mathcal X$). The right
finitistic $\mathcal X$-projective dimension of $R$ is
\begin{displaymath}
  \FPD_\mathcal{X}(R)= \Min\{n < \omega: \mathcal X_n = \mathcal X_{n+1}\}
\end{displaymath}
in case the set in the right is not empty, or $\infty$
otherwise. The right finitistic
$\mathcal Y$-injective dimension, $\FID_{\mathcal Y}(R)$, for a class
$\mathcal Y$ containing all injective modules is defined analogously. We will denote by
$\Proj$, $\Inj$, $\Flat$, $\Abs$, $\Cot$, $\PProj$ and $\PInj$ the
classes $\Modr{R}$ consisting of all projective, injective, flat,
absolutely pure, cotorsion, pure projective and pure injective modules
respetively.

\subsection{Locally split short exact sequences}

An epimorphism (resp. monomorphism) $f:M \rightarrow N$ in $\ModR$ is  \emph{locally split} if for each $x \in N$ (resp. $x \in M$)
there exists $g \colon N\rightarrow M$ such that $fg(x) = x$
(resp. $gf(x)=x$).

By \cite[Corollary 2]{Azumaya}, the morphism $f$ is a locally split
epimorphism if
and only if for each finite subset $F$ of $N$, there
exists $g\colon N \rightarrow M$ such that $fg(x)=x$ for each $x \in
F$. Moreover, by \cite[Proposition 11]{ZH}, if $f$ is a locally
split epimorphism and $N$ is countably generated, then $f$ is actually
split. We shall use the following property of
locally split epimorphisms:

\begin{lem}\label{l:LocallySplitEpi}
  Consider the following commutative diagram in $\ModR$
  \begin{displaymath}
    \begin{tikzcd}
      M \arrow{r}{f} \arrow{d}{h} & N \arrow{d}{g}\\
      K \arrow{r}{i} & L
    \end{tikzcd}
  \end{displaymath}
  in which $f$ and $g$ are locally split epimorphisms. Then so is $i$.
\end{lem}

\begin{proof}
  Let $x \in L$ and $\overline g\mathcolon L \rightarrow N$ be such
  that $g\overline g(x) = x$. Moreover, let
  $\overline f\mathcolon N \rightarrow M$ be such that
  $f\overline f(\overline g(x)) = \overline g(x)$. Then the morphism
  $\overline i:=h \overline f \overline g$ from $L$ to $K$ satisfy
  $i \overline i(x) = x$. Since $x$ is arbitrary, this means that $i$
  is locally split.
\end{proof}

Given a submodule $K$ of a
module $M$, if the inclusion $K \rightarrow M$ (resp. the projection
$M \rightarrow \frac{M}{K}$) is locally split, we shall say that $K$
(resp. $\frac{M}{K}$) is a \emph{locally split submodule}
(resp. \emph{quotient}) of $M$.  Given $\mathcal C$ a class of modules we shall denote by:
\begin{itemize}
\item $\Add(\mathcal C)$ the class of all modules that are isomorphic
  to a direct summand of a direct sum of modules in $\mathcal C$;

\item $\add(\mathcal C)$ the class of all modules that are isomorphic
  to a direct summand of a finite direct sum of modules in
  $\mathcal C$;

\item $\lse(\mathcal C)$ the class of all modules $M$ for which there
  exists a locally split epimorphism $A \rightarrow M$ with
  $A \in \Add(\mathcal C)$;

\item and $\lsm(\mathcal C)$ the class of all modules $M$ for which
  there exists a locally split monomorphism $M \rightarrow A$ with
  $A \in \Add(\mathcal C)$.
\end{itemize}

The following lemma recalls properties of the classes
$\lse(\mathcal C)$ and $\lsm(\mathcal C)$. The properties (1), (2) and
(4) are well known. The property (3) is an extension of a classical result for free
modules due to Villamayor (see \cite[Proposition 2.2]{Chase}):

\begin{lem}\label{l:PropertiesOfG}
  Let $\mathcal C$ be a class of modules.
  \begin{enumerate}
  \item The class $\lse(\mathcal C)$ is closed under locally split
    quotients and locally split submodules. Consequently,
    $\lsm(\mathcal C) \subseteq \lse(\mathcal C)$.

  \item Countably generated modules belonging to $\lse(\mathcal C)$ or
    $\lsm(\mathcal C)$ actually belong to $\Add(\mathcal C)$.

  \item If $\mathcal C$ consists of finitely generated modules, then
    $\lsm(\mathcal C)$ is closed under pure submodules.

  \item If $\mathcal C$ consists of finitely presented modules, then
    $\lse(\mathcal C)$ is closed under pure submodules.
  \end{enumerate}
\end{lem}

\begin{proof}
  (1) The class $\lse(\mathcal C)$ is clearly closed under locally
  split quotients. It is closed under locally split submodules as a
  consequence of \cite[Lemma 2.1]{Angeleri}.
 
  \noindent
  (2) Since a locally split epimorphism onto a countably generated
  module is split, a countably generated module belonging to
  $\lse(\mathcal C)$ actually belongs to $\Add(\mathcal C)$. As
  $\lsm(\mathcal C) \subseteq \lse(\mathcal C)$, countably generated
  modules in $\lsm(\mathcal C)$ belong to $\Add (\mathcal C)$ too.

\noindent
(3) Let $N \in \lsm(\mathcal C)$ and $K$ a pure submodule of
$N$. Since $N$ is isomorphic to a locally split submodule of a module
$A$ which is a direct summand of a direct sum of finitely generated
modules, we may assume that there exists a family
$\{N_i\mid i \in I\}$ of finitely generated modules such that $N$ is a
locally split submodule of $\bigoplus_{i \in I}N_i$.

\noindent  
We prove that $K$ is a locally split submodule of $N$ which implies,
by \cite[Proposition 1.3]{Zimm}, that $K \in \lsm(\mathcal C)$. Let
$U$ be a finitely generated submodule of $K$. There exists a finite
set $J \subseteq I$ and a commutative diagram
\[\begin{tikzcd}
  0 \arrow{r} & U \arrow[hook]{r} \arrow[hook]{d} & \bigoplus_{j \in
    J}N_j \arrow[hook]{d}
  \arrow{r}{p'} & C' \arrow{r} \arrow{d}{q} & 0\\
  0 \arrow{r} & K \arrow[hook]{r} & \bigoplus_{i \in I}N_i
  \arrow{r}{p} & C \arrow{r} & 0.
\end{tikzcd}\]
Since $K$ is pure in $\bigoplus_{i \in I}N_i$ and $C'$ is finitely
presented, there exists a morphism
$r\colon C' \rightarrow \bigoplus_{i \in I}N_i$ such that $pr=q$. By
\cite[7.16]{Wis} there exists
$\overline g\mathcolon \bigoplus_{j \in J}N_j \rightarrow K$ such that
$\overline g(u)=u$ for each $u \in U$. This $\overline g$ extends to a
morphism $g'\colon \bigoplus_{i \in I}N_i \rightarrow K$ whose
restriction to $N$, $g\colon N \rightarrow K$, trivially satisfies
$g(u)=u$ for each $u \in U$. Then $K$ is a locally split submodule of
$N$.

\noindent
(4). This is \cite[Proposition 2.3]{Angeleri}.
\end{proof}

A short exact sequence
$$0\to K\stackrel{f}{\to} M\stackrel{g}{\to}N\to 0$$ is said to be
locally split if $f$ is a locally split monomorphism and $g$ is a
locally split epimorphism. Contrary to the case of pure exact
sequences, there exist locally split epimorphisms and locally split
monomorphisms such that the corresponding short exact sequence is not
locally split (see \cite[Example 13]{ZH}). However, if in a short exact sequence $$0\to
K\stackrel{f}{\to} M\stackrel{g}{\to}N\to 0,$$ $g$ is locally split
and $M$ is a direct sum of countably generated modules, then the
sequence is locally split by \cite[Proposition 12]{ZH}. This is the
situation of our Theorem \ref{Intr.Thm 1}.

\subsection{Complete cotorsion pairs}

A pair of classes $(\A,\B)$ in an abelian category $\G$ is called a
\emph{cotorsion pair} if $\A^\perp=\B$ and $^{\perp} \B=\A$, where,
for a given class of objects $\C$, the right orthogonal class
$\C^\perp$ is defined as the class of objects $M$ such that
$\Ext^1_{\G}(C,M)=0$ for each object $C\in \C$. Similarly, the left
orthogonal class $^\perp \C$ is defined. The cotorsion pair is called
\emph{hereditary} if $\Ext^i_{\G}(A,B)=0$ for all $A\in \A$ and
$B\in \B$, and $i\geq 1$. Finally we say that the cotorsion pair is
\emph{complete} when it has \emph{enough injectives} and \emph{enough
  projectives}. This means that for each $M\in \G$ there exist exact
sequences $0\to M\to A\to B\to 0$ (enough injectives) and
$0\to B'\to A'\to M\to 0$ (enough projectives), where $A,A'\in \A$ and
$B,B'\in \B$.

\subsection{Chain complexes of modules}
Let $\G$ be an abelian category, we denote by $\Ch(\G)$ the category
of unbounded chain complexes of objects in $\G$, i.e. complexes $G$ of
the form
$$\cdots\to G_{n+1}\xrightarrow{d^G_{n+1}}G_n\xrightarrow{d^G_{n}}
G_{n-1}\to\cdots.$$
We will denote by $Z_n G$ the \emph{$n$ cycle of} $G$,
i.e. $Z_nG=\Ker(d^G_n)$. Given a chain complex $G$ the
\emph{$n^{th}$-suspension of $G$}, $\Sigma^n G$, is the complex
defined as $(\Sigma^n G)_k=G_{k-n}$ and
$d^{\Sigma^n G}_k=(-1)^n d_{k-n}^G$. And for a given object $A\in \G$,
the \emph{$n$-disk} complex is the complex with the object $A$ in the
components $n$ and $n-1$, $d_n$ as the identity map, and 0
elsewhere. Given a covariant functor $F\colon\G \rightarrow H$ between
abelian categories, we shall denote by $F(G)$ the complex
\begin{displaymath}
  \cdots\to
  F(G_{n+1})\xrightarrow{F(d_{n+1}^G)}F(G_n)\xrightarrow{F(d_{n}^G)} F(G_{n-1})\to\cdots
\end{displaymath}
If $F$ is contravariant, we define $F(G)$ in a similar way.

In case $\G=\Modr R$, we will denote $\Ch(\G)$ simply by $\Ch(R)$. An
acyclic complex $G$ in $\Ch(R)$ consisting of projective modules
(resp. injective modules) is said to be \emph{totally acyclic} if
$\Hom_R(G,P)$ (resp. $\Hom_R(I,G)$) is acyclic for each projective
module $P$ (resp. injective module $I$).  For every class $\C$ of
modules, we shall consider the following classes of complexes of
modules:

\begin{itemize}
\item $\dw\C$ is the class of all complexes $X\in\Ch(R)$ such that
  $X_n\in\C$ for all $n\in\bbZ$.  $\Ch(\C)$ will denote the full
  subcategory of $\Ch(R)$ with objects in $\dw\C$.
\item $\ex\C$ is the class of all acyclic complexes in $\dw\C$.
\item $\widetilde{\C}$ is the class class of all complexes
  $X\in \ex\C$ with the cycles $Z_nX$ in $\C$ for all $n\in \bbZ$.
  % $\Ch_{ac}(\C)$ will denote the full subcategory of $\Ch(R)$ with
  % objects in $\tilde\C$.
  % \item $\Cot$ denotes the class of cotorsion modules and $\FpInj$
  %   the class of absolutely pure modules.
  % \item $\Proj$ denotes the class of projective modules and $\Flat$
  %   the class of flat modules.

\item If $(\A, \B)$ is a cotorsion pair in $\Modr{R}$, then $\dg\A$ is
  the class of all complexes $X\in \dw\A$ such that any morphism
  $f\mathcolon X\to Y$ with $Y\in\tilclass \B$ is null
  homotopic. Since $\Ext^1_{R}(A_n, B_n)=0$ for every $n\in \bbZ$, a
  well known formula shows that $\dg\A={}^\perp{} \tilclass
  \B$.
  Similarly, $\dg\B$ is the class of all complexes $Y \in \dw\B$ such
  that any morphism $f\mathcolon X\to Y$ with $X\in\tilde \A$ is null
  homotopic. Hence $\dg\B=\tilde\A{}^\perp{}$.
\end{itemize}
If $\mathcal C$ is a class of complexes of modules, we shall denote by
$\textrm{Z}(\mathcal C)$ the class of all modules isomorphic to a
cycle of some complex belonging to $\mathcal C$.

\subsection{$\C$-periodic modules.}
We are interested in periodic modules with respect to a class of modules.

\begin{defn}
  Let $\mathcal C$ be a class of modules and $M$ a module. We say that
  $M$ is $\mathcal C$-periodic (resp. pure $\mathcal C$-periodic) if
  there exists an exact sequence (resp. a pure exact sequence)
$$0\to M\to C\to M\to 0$$
with $ C \in \mathcal C$.
\end{defn}

These modules are related with Gorenstein modules. Recall that a
module $M$ is \emph{Gorenstein projective} (resp. \emph{Gorenstein
  injective}) if there exists a totally acyclic complex of projective
modules (resp. injective modules) such that $M \cong Z_0G$. And $M$ is
said to be \emph{strongly Gorenstein projective} (resp. \emph{strongly
  Gorenstein injective}) if it is Gorenstein projective and
$\Proj$-periodic (resp. Gorenstein injective and $\Inj$-periodic), see
\cite[Proposition 2.9]{BennisMahdou}. By \cite[Theorem
2.7]{BennisMahdou}, each Gorenstein projective (resp. Gorenstein
injective) module is a direct summand of a strongly Gorenstein
projective (resp. strongly Gorenstein injective) module.

Analogously, a module $M$ is said to be Gorenstein flat if there
exists an acyclic complex $G$ consisting of flat modules such that
$G \otimes_R I$ is exact for each injective left $R$-module $I$ and
$M \cong Z_0G$. And $M$ is strongly Gorenstein flat if it is
$\Flat$-periodic and Gorenstein flat. By \cite[Theorem
3.5]{BennisMahdou}, each Gorenstein flat module is a direct summand of
a strongly Gorenstein flat module.

The relationship between Gorenstein modules and strongly Gorenstein
modules  observed by Bennis and Mahdou can be extended to cycles of certain chain complexes and
periodic modules, as we show in the following proposition. This approach was
used by Christensen and Holm \cite{CH} and by Fu and Herzog
\cite{FH16} in the special case of flat $\Proj$-periodic modules. We present now the general formulation that appears in \cite{EFI}. 

\begin{prop}\label{P:cycles}(\cite[Proposition 1 and Proposition
  2]{EFI}) Let $\C$ be a class of modules closed under direct sums or under direct
  products and let $\D$ be a class closed under direct summands. The
  following are equivalent:
  \begin{enumerate}
  \item Every cycle of an acyclic complex with components in $\C$
    belongs to $\D$, that is $\Z(\ex(\mathcal C)) \subseteq \D$.
  \item Every $\C$-periodic module belongs to $\D$.
  \end{enumerate}
\end{prop}

Let $\mathcal C$ be a class of modules. One of our main concerns is
when $\mathcal C$-periodic modules are trivial, in the sense that they
belong to the initial class $\mathcal C$. The preceeding result gives
us a procedure to construct non trivial $\mathcal C$-periodic modules:
we only have to find a complex in $\ex \mathcal C$ whose cycles do not
lie in $\mathcal C$. This complex exists if $\mathcal C$ is a
generating and cogenerating class in $\rmod{R}$.

\begin{cor}\label{p:ConstructionPeriodic}
  Let $\mathcal C$ be a generating and cogenerating class of modules
  closed under direct sums or direct products. Let $\mathcal D$ be a
  class of modules closed under direct summands which is not equal to
  the whole category $\Modr R$. Then there exist periodic modules not
  belonging to $\mathcal D$.
\end{cor}

\begin{proof}
  Let $M$ be a module not belonging to $\D$. Since $\mathcal C$ is
  generating and cogenerating, there exists a complex
  $C \in \ex \mathcal C$ such that $M \cong Z_0 C$. Then, if
  $\mathcal C$ is closed under direct sums (resp. direct products),
  $\bigoplus_{n \in \mathbb Z}Z_nC$ (resp.
  $\prod_{n \in \mathbb Z}Z_nC$) is a $\mathcal C$-periodic module
  not belonging to $\mathcal D$.
\end{proof}

\section{Locally split monomorphisms and generalized periodic modules}

% \begin{lemma}
%   Let $G$ be a module which is a direct sum of countably generated
%   modules and $M$ a locally split submodule of $G$. Then, for each
%   cardinal $\lambda$ and each $\lambda$-generated submodule $K$ of
%   $M$, there exists a locally split submodule $L$ of
%   $\bigoplus_{i \in I}G_i$ containing $K$ that is
%   $\Max\{\aleph_0,\lambda\}$-generated.
% \end{lemma}

% \begin{proof}
%   This is essentially proved in \cite[Lemma 1.2]{Simson}; we include
%   the proof for completness. Let $\{x_\alpha:\alpha < \lambda\}$ be
%   a generating system of $K$. For any $\alpha < \lambda$, there
%   exists $f_\alpha\colon G \rightarrow M$ such that
%   $f_\alpha(x_\alpha) = x_\alpha$. Since $x_\alpha$ is contained in
%   a countably generated direct summand of $G$, we may assume that
%   $\Img f_\alpha$ is countably generated. Now let
%   $K_1 = \sum_{\alpha < \lambda}\Img f_\alpha$ which is
%   $\Max\{\aleph_0,\lambda\}$-generated. Note that $K_1$ satisfy that
%   for each $x \in K$ there exists a morphism $f:G \rightarrow K_1$
%   such that $f(x)=x$. Now we can repeat this argument recursively to
%   obtain a chain of $\Max\{\aleph_0,\lambda\}$-generated submodules
%   of $M$, $K_1 \leq K_2 \leq \cdots$, such that $K \leq K_1$ and for
%   each nonzero natural number $n$ and $x \in K_n$, there exists
%   $f\colon G \rightarrow K_{n+1}$ such that $f(x)=x$. Then the proof
%   is completed by taking $L = \bigcup_{n < \omega}K_n$.
% \end{proof}

This section is devoted to prove Theorem \ref{Intr.Thm 1}. We begin
with a technical lemma. In general, if $f_1$ and $f_2$ are split monomorphisms in the commutative diagram
\begin{displaymath}
  \begin{tikzcd}
    K' \arrow{r}{i} \arrow{d}{f_1} & M' \arrow{d}{f_2}\\
    K \arrow{r}{k} & M
  \end{tikzcd}
\end{displaymath}then the splittings of $f_1$
and $f_2$ need not make the diagram commutative. The following lemma
constructs a splitting which makes the diagram commutative in some
particular cases.

\begin{lem}\label{l:lemma1}
  Consider the following diagram in $\ModR$ with exact rows
  \begin{displaymath}
    \begin{tikzcd}
      0 \arrow{r} & K' \arrow[hook]{r}{i} \arrow{d}{f_1} & M'
      \arrow{r}{j} \arrow{d}{f_2}& L'
      \arrow{r} \arrow{d}{f_3} & 0\\
      0 \arrow{r} & K \arrow{r}{k} & M
      \arrow{r}{l} & L \arrow{r} & 0\\
    \end{tikzcd}
  \end{displaymath}
  such that $f_1$, $f_3$ and $k$ are splitting monomorphisms. Then for
  any splitting $\overline f_1$ of $f_1$, there exists a splitting
  $\overline f_2$ of $f_2$ such that $i\overline f_1=\overline f_2k$.
\end{lem}

\begin{proof}
  Let $\overline f_3$ be a splitting of $f_3$ and $\overline k$ a
  splitting of $k$. Then
  $\overline i:=\overline f_1 \overline k f_{2}$ is a splitting of $i$
  and, consequently, there exists $\overline j$ a splitting of $j$.

\noindent
Now let $l'$ be a splitting of $l$ and note that
$L = \Img f_3 \oplus N$ for some $N$. Denoting by $\overline l$ the
direct sum of the morphisms
$f_2\overline j (\overline f_3 \rest \Img f_3)\mathcolon \Img f_3
\rightarrow M$
and $l' \rest N\mathcolon N \rightarrow M$, we obtain a splitting of
$l$ satisfying $\overline l f_3 = f_2\overline j$.

\noindent
Finally, notice that $M = \Img k \oplus \Img \overline l$. Define
$\overline f_{2}$ as the direct sum of the morphisms
$i\overline f_1\overline k\mathcolon\Img k \rightarrow M'$ and
$\overline j \overline f_3 l\mathcolon\Img \overline l \rightarrow
M'$.
It is easy to see that $\overline f_{2}$ is a splitting of
$f_{2}$ satisfying $i \overline f_1 = \overline f_2 k$.
\end{proof}

In general, the union of a chain of direct summands of a module is not
a direct summand. The following lemma shows a situation in which the
union of a continuous chain of direct summands is a direct summand.

\begin{lem}\label{l:split}
  Let $G$ be a module and $M$ a submodule of $G$. Suppose that there
  exist an ordinal $\kappa$, a continuous chain $\{M_\alpha\mid \alpha < \kappa\}$ of submodules of $M$,
   with
  $M = \bigcup_{\alpha < \kappa}M_\alpha$, and a continuous chain $\{G_\alpha\mid \alpha < \kappa\}$ of
  submodules of $G$,  such that:
  \begin{enumerate}[label=(\alph*)]
  \item $\bigcup_{\alpha < \kappa}G_\alpha$ is a direct summand of
    $G$;

  \item $G_\alpha$ is a direct summand of $G_{\alpha+1}$ for each
    $\alpha < \kappa$;

  \item $M_\alpha=M_{\alpha+1}\cap G_\alpha$ for each
    $\alpha < \kappa$;

  \item $M_0$ is a direct summand of $G_0$ and
    $\frac{M_{\alpha+1}+G_\alpha}{G_\alpha}$ is a direct summand of
    $\frac{G_{\alpha+1}}{G_\alpha}$ for each $\alpha < \kappa$.
  \end{enumerate}
  Then $M$ is a direct summand of $G$.
\end{lem}

\begin{proof}
  Denote by $f_\alpha\mathcolon M_\alpha \rightarrow G_\alpha$ the
  inclusion for each $\alpha < \kappa$. We are going to construct, for
  each $\alpha < \kappa$, a morphism
  $\overline f_\alpha\mathcolon G_\alpha \rightarrow M_\alpha$ such
  that $\overline f_\alpha f_\alpha = 1_{M_\alpha}$ and
  $i_{\gamma\alpha} \overline f_\gamma = \overline f_\alpha
  j_{\gamma\alpha}$
  for each $\gamma < \alpha$, where
  $i_{\gamma\alpha}\mathcolon M_\gamma \rightarrow M_{\alpha}$ and
  $j_{\gamma\alpha}\mathcolon G_\gamma \rightarrow G_{\alpha}$ are the
  inclusions. Then the direct limit map of the $\overline f_\alpha$'s
  is a splitting of the inclusion
  $M \hookrightarrow \bigcup_{\alpha < \kappa}G_\alpha$, which implies
  that $M$ is a direct summand of $G$ by (a).
  
\noindent
We shall make the construction recursively on $\alpha$. If $\alpha=0$,
take a splitting $\overline f_0\mathcolon G_0 \rightarrow M_0$ of
$f_0$. If $\alpha$ is limit, let $\overline f_\alpha$ be the direct
limit of the system $\{\overline f_\gamma\mid \gamma < \alpha\}$.
Finally, suppose that $\alpha$ is succesor, say $\alpha = \mu+1$.
Using the snake lemma we can construct the following commutative
diagram,
\[\begin{tikzcd}
  0 \arrow{r} & M_\mu \arrow[hook]{r}{f_\mu}
  \arrow[hook]{d}{i_{\mu\mu+1}} & G_\mu \arrow{r}{g_\mu}
  \arrow[hook]{d}{j_{\mu \mu+1}}& \frac{G_\mu}{M_\mu}
  \arrow{r} \arrow{d}{k_\mu} & 0\\
  0 \arrow{r} & M_{\mu+1} \arrow[hook]{r}{f_{\mu+1}} \arrow{d}{p_\mu}
  & G_{\mu+1}
  \arrow{r}{g_{\mu+1}} \arrow{d}{q_\mu} & \frac{G_{\mu+1}}{M_{\mu+1}} \arrow{r} \arrow{d}{r_\mu}& 0\\
  0 \arrow{r} & \frac{M_{\mu+1}}{M_\mu} \arrow{r}{h_{\mu}} &
  \frac{G_{\mu+1}}{G_\mu} \arrow{r}{l_\mu} &
  \frac{G_{\mu+1}}{M_{\mu+1}+G_\mu} \arrow{r} & 0
\end{tikzcd}\]
in which $k_\mu$ is monic, since $M_{\mu+1} \cap G_\mu = M_\mu$ by (c)
and, consequently, $h_\mu$ is monic too. But
$h_\mu(x+M_\mu) = x+G_\mu$ for each $x \in M_{\mu+1}$, which means
that $\Img h_\mu = \frac{M_{\mu+1}+G_\mu}{G_\mu}$. By (d), $h_\mu$
splits. Since $j_{\mu\mu+1}$ splits by (b) and $f_\mu$ splits by induction
hyphotesis, we can apply Lemma \ref{l:lemma1} to the splitting
$\overline f_\mu$ of $f_\mu$ to construct a spliting $\overline f_{\mu+1}$ of $f_{\mu+1}$ such that
$\overline f_{\mu+1}j_{\mu\mu+1}=i_{\mu\mu+1}\overline f_{\mu}$. This
concludes the construction.
\end{proof}

The proof of Theorem \ref{Intr.Thm 1} relies on the following lemmas.

\begin{lem}\label{l:locally-split-epi}
  Let $\lambda$ be an infinite cardinal and
  $g\mathcolon G \rightarrow M$ be a locally split epimorphism. Then
  for each $\leq \lambda$-generated submodule $C$ of $M$, there exists a
  $\leq \lambda$-generated submodule $D$ of $G$ such that $C\subseteq g(D)$
  and $g\rest D\colon D \rightarrow g(D)$ satisfies the following
  property: for each $x \in g(D)$, there exists $f\mathcolon M \rightarrow G$
  such that $gf(x)=x$ and $fg(D) \leq D$.

  If, in addition, $G=\oplus_{i \in I}G_i$ for a family
  $\{G_i:i \in I\}$ of countably generated modules, then there exists
  a subset $J$ of $I$ of cardinality less than or equal to $\lambda$
  such that $D=\oplus_{j \in J}D_j$.
\end{lem}

\begin{proof}
  We are going to construct a chain $\{D_n:n < \omega\}$ of $\leq \lambda$-generated submodules
  of $G$  with $C \leq g(D_0)$, and a
  chain $\{H_n:n < \omega\}$ of subsets of $\Hom_R(M,G)$ such that
  $|H_n|\leq \lambda$, verifying the following properties for every $n < \omega$:
  \begin{enumerate}[label=(\alph*)]
  \item given $x \in g(D_n)$ there exists $f \in H_n$ with $gf(x)=x$,

  \item and $fg(D_n) \leq D_{n+1}$ for each $f \in H_n$.
  \end{enumerate}
  We shall make the construction recursively on $n$. For $n = 0$ let
  $D_0$ be a $\lambda$-generated submodule of $G$ such that
  $C \leq g(D_0)$ and take a generating set
  $\{x_\alpha^0:\alpha < \lambda\}$ of $g(D_0)$. Since $g$ is locally
  split, for each finite subset $\Gamma$ of $\lambda$ there exists, by \cite[Corollary 2]{Azumaya}, a
  morphism $f^0_\Gamma \dd M \rightarrow G$ such that
  $gf^0_\Gamma(x^0_\gamma) = x^0_\gamma$ for each $\gamma \in \Gamma$.
  Then the set
  $H_0:=\{f^0_\Gamma:\Gamma \subseteq \lambda \textrm{ is finite}\}$
  has cardinality less than or equal to $\lambda$ and for each
  $x \in g(C)$, there exists $f \in H_0$ with $gf(x)=x$.

\noindent
Assume that we have constructed $D_n$ and $H_n$ for some
$n<\omega$. Let $D_{n+1}=\sum_{f \in H_n}f(g(D_n))$, which is a
$\leq \lambda$-generated submodule (since $|H_n| \leq \lambda$ and $D_n$ is
$\leq \lambda$-generated), and satisfies $fg(D_n) \leq D_{n+1}$ for
each $f \in H_n$.  Let $\{x_\alpha^{n+1}:\alpha < \lambda\}$ be a
generating system of $g(D_{n+1})$. Again by \cite[Corollary 2]{Azumaya} since $g$ is locally split, there
exists,  for each $\Gamma \subseteq \lambda$ finite, a morphism
$f^{n+1}_\Gamma\dd M \rightarrow G$ such that
$gf^{n+1}_\Gamma(x^{n+1}_\gamma) = x^{n+1}_\gamma$ for each
$\gamma \in \Gamma$. Set
$H_{n+1}=\{f^{n+1}_\Gamma:\Gamma \subseteq \lambda \textrm{ is
  finite}\}\cup H_n$. Then $H_{n+1}$
has cardinality less than or equal to $\lambda$ and satisfies
that for each $x \in g(D_{n+1})$ there exists $f \in H_{n+1}$ with
$gf(x)=x$. This concludes the construction.

  \noindent
  Now take $D=\bigcup_{n < \omega}D_n$, which is
  $\leq \lambda$-generated. Given $x \in g(D)$, there exists $n < \omega$
  such that $x \in g(D_n)$. By (a) there exists $f \in H_n$ such that
  $fg(x)=x$. Moreover, this $f$ satisfies that $fg(D) \leq D$ since,
  given $k < \omega$ with $k \geq n$, $fg(D_k) \leq D_{k+1} \leq D$ by
  (b), because $f \in H_k$. And, if $k < n$, then $D_k \leq D_n$ which
  implies that $f(D_k) \leq D_{n+1}$ by (b) again. This finishes the
  proof of the first part of the lemma.

\noindent
In order to prove the last statement simply note that, if $G$ is a
direct sum of countably generated modules, $G=\oplus_{i \in I}G_i$,
then, in the previous construction, $D_n$ can be taken of the form
$\oplus_{i \in I_n}G_i$ for some subset $I_n$ of $I$ satisfying
$I_n \subseteq I_{n+1}$ for each $n < \omega$.
\end{proof}

\begin{lem}\label{l:SuccesorStep}
  Let
  \[\begin{tikzcd}
    0 \arrow{r} & M \arrow[hook]{r} & G \arrow{r}{g} & M \arrow{r} & 0
  \end{tikzcd}\]
  be a short exact sequence in $\ModR$ such that $g$ is locally split and $G$ is
  a direct sum of a family $\{G_i\mid i \in I\}$ of countably
  generated modules. Then, for each countably generated submodule $K$
  of $M$, there exists a countable subset $J$ of $I$ such that
  $K \leq \oplus_{j \in J}G_j$,
  $g(\oplus_{j \in J}G_j) = M \cap (\oplus_{j \in J}G_j)$ and the
  sequence
  \[\begin{tikzcd}
    0 \arrow{r} & M \cap (\oplus_{j \in J}G_j) \arrow[hook]{r}&
    \oplus_{j \in J}G_j \arrow{r}{g} & g(\oplus_{j \in J}G_j)
    \arrow{r} & 0
  \end{tikzcd}\] is split exact.
\end{lem}

\begin{proof}
  We are going to construct two chains of countable subsets of $I$,
  $\{I_n\mid n < \omega\}$ and $\{J_n\mid n < \omega\}$, such that
  $K \leq \oplus_{i \in I_0}G_i$, and a chain of subsets of
  $\Hom_R(M,G)$, $\{H_n:n < \omega\}$, satisfying, for each
  $n < \omega$:
  \begin{enumerate}[label=(\alph*)]
  \item $I_n \leq J_n \leq I_{n+1}$;

  \item
    $M \cap \left(\oplus_{i \in I_n}G_i\right) \leq g(\oplus_{i \in
      J_n}G_i) \leq \oplus_{i \in I_{n+1}}G_i$,

  \item and for each $x \in g(\oplus_{i \in I_n}G_i)$ there exists
    $f \in H_n$ such that $gf(x)=x$ and
    $f(g(\oplus_{i \in I_n}G_i)) \leq \oplus_{i \in I_{n}}G_i$.
  \end{enumerate}
  We shall make the construction recursively on $n$. For $n=0$ let
  $L_0$ be a countable subset of $I$ such that $K \leq \oplus_{i \in
    L_0}G_i$. Now apply
  Lemma \ref{l:locally-split-epi} to get a countable subset $I_0$ of
  $I$ containing $L_0$ such that $g(\oplus_{i \in L_0}G_i) \leq g(\oplus_{i \in I_0}G_i)$, and a subset
  $H_0$ of $\Hom_R(M,G)$ satisfying that for each
  $x \in g(\oplus_{i \in I_0}G_i)$ there exists $f \in H_0$ with
  $gf(x)=x$ and
  $f(g(\oplus_{i \in I_0}G_i)) \leq \oplus_{i \in I_{0}}G_i$. In
  particular, $g \upharpoonright \oplus_{i \in I_0}G_i$ is locally
  split and, as $g(\oplus_{i \in I_0}G_i)$ is countably generated,
  $g \upharpoonright \oplus_{i \in I_0}G_i$ is actually split. This
  means that $M \cap (\oplus_{i \in I_0}G_i)$ is a direct summand of
  $\oplus_{i \in I_0}G_i$ and, consequently, it is countably
  generated. Therefore, there exists a countable subset $J_0$ of
  $I$ containing $I_0$ such that
  $M \cap (\oplus_{i \in I_0}G_i) \leq g(\oplus_{i \in J_0}G_i)$. This
  concludes case $n=0$.

\noindent
Now assume that we have constructed $I_n$ and $J_n$ for some
$n < \omega$, and let us construct $I_{n+1}$, $J_{n+1}$ and
$H_{n+1}$. Since $g(\oplus_{i \in J_n}G_i)$ is countably generated
there exists a countable subset $L_n$ containing $J_n$ such that
$g(\oplus_{i \in J_n}G_i) \leq \oplus_{i \in L_n}G_i$. Then we can
apply again Lemma \ref{l:locally-split-epi} to find a countable subset
$I_{n+1}$ of $I$ containing $L_n$ such that
$g(\oplus_{i \in L_n}G_i) \leq g(\oplus_{i \in I_{n+1}}G_i)$, and a
subset $H_{n+1}$ of $\Hom_R(M,G)$ satisfying that for each
$x \in g(\oplus_{i \in I_{n+1}}G_i)$ there exists $f \in H_{n+1}$ with
$gf(x)=x$ and
$f(g(\oplus_{i \in I_{n+1}}G_i)) \leq \oplus_{i \in I_{n+1}}G_i$. In
particular, $g \upharpoonright \oplus_{i \in I_{n+1}}G_i$ is locally
split and, as $g(\oplus_{i \in I_{n+1}}G_i)$ is countably generated,
$g \upharpoonright \oplus_{i \in I_{n+1}}G_i$ is actually split. This
means that $M \cap (\oplus_{i \in I_{n+1}}G_i)$ is a direct summand of
$\oplus_{i \in I_{n+1}}G_i$ and so it is countably
generated. Consequently, there exists a countable subset $J_{n+1}$ of
$I$ containing $I_{n+1}$ such that
$M \cap (\oplus_{i \in I_{n+1}}G_i) \leq g(\oplus_{i \in
  J_{n+1}}G_i)$. This concludes the construction.

\noindent
Finally, let $J=\bigcup_{n<\omega}J_n=\bigcup_{n<\omega}I_n$. Then, by
(b), $g(\oplus_{j \in J}G_j) = M \cap \oplus_{j \in J}G_j$ so that the
sequence
\[\begin{tikzcd}
  0 \arrow{r} & M \cap (\oplus_{j \in J}G_j) \arrow[hook]{r}&
  \oplus_{j \in J}G_j \arrow{r}{g} & g(\oplus_{j \in J}G_j) \arrow{r}
  & 0
\end{tikzcd}\]is
exact. Moreover, as a consequence of (c),
$g \rest \oplus_{j \in J}G_j$ is locally split and, since
$g(\oplus_{j \in J}G_j)$ is countably generated, it is actually
split. This concludes the proof.
\end{proof}

We are now in position to prove Theorem \ref{Intr.Thm 1}.

\begin{thm}\label{t:SplittingSequence}
  Any short exact sequence
  \[\begin{tikzcd}
    0 \arrow{r} & M \arrow[hook]{r} & G \arrow{r}{g} & M \arrow{r} & 0
  \end{tikzcd}\]
  in which $g$ is locally split and $G$ is a direct sum of countably
  generated modules is split.
\end{thm}

\begin{proof}
  Write $G=\oplus_{i \in I}C_i$ as a direct sum of countably generated
  modules and fix $\{x_\alpha\mid \alpha < \kappa\}$ a generating
  system of $M$ for some cardinal $\kappa$. Our aim is to use Lemma \ref{l:split} with the
  submodule $M$ of $G$.  In order to construct the chain of submodules
  $\{M_\alpha\mid \alpha < \kappa\}$ and
  $\{G_\alpha\mid \alpha < \kappa\}$ satisfying the hypothesis of this
  lemma, we are going to apply recursively Lemma
  \ref{l:SuccesorStep}. Actually, we are going to construct a
  continuous chain of subsets of $I$,
  $\{I_\alpha\mid \alpha < \kappa\}$, such that
  $G_\alpha = \oplus_{i \in I_\alpha}C_i$.

\noindent
Let us make the construction of the chain $\{M_\alpha\mid \alpha < \kappa\}$ of submodules of $M$,
 and of the chain $\{I_\alpha\mid \alpha < \kappa\}$ of subsets of $I$ satisfying Lemma \ref{l:split}
and, for each $\alpha < \kappa$,

\begin{enumerate}[label=(\alph*)]
\item $x_\alpha \in M_\alpha$; 

\item $M_\alpha = M \cap G_\alpha = g(G_\alpha)$,

\item and $M_\alpha$ is a direct summand of $G_\alpha$.
\end{enumerate}
We shall proceed recursively on $\alpha$.

\noindent
For $\alpha = 0$, take $I_0$ the countable subset obtained in Lemma
\ref{l:SuccesorStep} for the countably generated submodule $Rx_0$ of
$M$, and set $M_0 = M \cap (\oplus_{i \in I_0}C_i)$ and
$G_0 = \oplus_{i \in I_0}C_i$. Notice that, as a consequence of Lemma
\ref{l:SuccesorStep}, $M_0$ is a direct summand of $G_0$.

\noindent
Let $\alpha$ be a nonzero ordinal such that we have made the
construction for each ordinal smaller than $\alpha$. If $\alpha$ is
limit, simply take $M_\alpha = \bigcup_{\gamma < \alpha}M_\gamma$,
$I_\alpha = \bigcup_{\gamma < \alpha}I_\gamma$ and
$G_\alpha = \oplus_{i \in I_\alpha}C_i$.

\noindent
Finally, assume that $\alpha$ is succesor, say $\alpha = \mu+1$. We
can construct, using the snake lemma, the following commutative
diagram with exact rows, in which $f_\mu$ is the inclusion, $g_\mu$ is
the restriction of $g$, and $p_\mu$ and $q_\mu$ are projections:
\[\begin{tikzcd}
  0 \arrow{r} & M_\mu \arrow[hook]{r}{f_\mu} \arrow[hook]{d}{i_\mu} &
  G_\mu \arrow{r}{g_\mu} \arrow[hook]{d}{j_\mu}& M_\mu
  \arrow{r} \arrow[hook]{d}{i_\mu} & 0\\
  0 \arrow{r} & M \arrow[hook]{r} \arrow{d}{p_\mu} & G
  \arrow{r}{g} \arrow{d}{q_\mu} & M \arrow{r} \arrow{d}{p_\mu}& 0\\
  0 \arrow{r} & \frac{M}{M_\mu} \arrow{r}{\hat f_{\mu}} &
  \frac{G}{G_\mu} \arrow{r}{\hat g_\mu} & \frac{M}{M_{\mu}} \arrow{r}
  & 0
\end{tikzcd}\]
Note that $f_\mu$ and $j_\mu$ are split, so that $i_\mu$ is split
too. Moreover, by Lemma \ref{l:LocallySplitEpi}, $\hat g_\mu$ is
locally split.

\noindent
Now, since $\hat f_\mu(x+M_\mu) = x+G_\mu$ for each $x \in M$,
$\Img \hat f_\mu = \frac{M+G_\mu}{G_\mu}$ and, consequently, we have a
short exact sequence
\[\begin{tikzcd}
  0 \arrow{r} & \frac{M+G_\mu}{G_\mu} \arrow[hook]{r} &
  \frac{G}{G_\mu} \arrow{r}{\hat f_\mu \hat g_\mu}&
  \frac{M+G_\mu}{G_\mu} \arrow{r}& 0
\end{tikzcd}\]
in which $\hat f_\mu \hat g_\mu$ is a locally split epimorphism and
$\frac{G}{G_\mu}$ is the direct sum of the family of countably
generated modules
$\left\{\frac{C_i+G_\mu}{G_\mu}\mid i \in I-I_\mu\right\}$. Then we
can apply Lemma \ref{l:SuccesorStep} to the countably generated
submodule $R(x_{\mu+1}+G_\mu)$ of $\frac{M+G_\mu}{G_\mu}$, to get a
countable subset $J_\mu$ of $I-I_\mu$ such that $x_{\mu+1}+G_\mu \in \bigoplus_{i \in
      J_\mu}\frac{C_i+G_\mu}{G_\mu}$, 
\begin{displaymath}
  \frac{M+G_\mu}{G_\mu} \bigcap \left(\bigoplus_{i \in
      J_\mu}\frac{C_i+G_\mu}{G_\mu}\right) = \hat f_\mu \hat g_\mu
  \left(\bigoplus_{i \in J_\mu}\frac{C_i+G_\mu}{G_\mu}\right)
\end{displaymath}
and the morphism
$\hat f_\mu \hat g_\mu \rest \left( \bigoplus_{i \in
    J_\mu}\frac{C_i+G_\mu}{G_\mu}\right)$
from $\bigoplus_{i \in J_\mu}\frac{C_i+G_\mu}{G_\mu}$ to
$\hat f_\mu \hat g_\mu\left(\bigoplus_{i \in
    J_\mu}\frac{C_i+G_\mu}{G_\mu}\right)$
is split exact. Set $I_{\mu+1} = I_{\mu} \cup J_{\mu}$,
$G_{\mu+1} = \oplus_{i \in I_{\mu+1}}C_i$ and
$M_{\mu+1}=g(G_{\mu+1})$. It is easy to see that
$g(G_{\mu+1}) = M \cap G_{\mu+1}$ and that $x_{\mu+1} \in \oplus_{i
  \in I_{\mu+1}}C_i$.

\noindent
Now we see that $\frac{M_{\mu+1}+G_\mu}{G_\mu}$ is a direct summand of
$\frac{G_{\mu+1}}{G_\mu}$ and $M_{\mu+1}$ is a direct summand of
$G_{\mu+1}$ to finish the proof. Applying again the snake lemma, we
get a commutative diagram with exact rows
\begin{displaymath}
  \begin{tikzcd}
    0 \arrow{r} & M_\mu \arrow[hook]{r}{f_\mu} \arrow[hook]{d} & G_\mu
    \arrow{r}{g_\mu} \arrow[hook]{d}& M_\mu
    \arrow{r} \arrow[hook]{d} & 0\\
    0 \arrow{r} & M_{\mu+1} \arrow[hook]{r}{f_{\mu+1}} \arrow{d} &
    G_{\mu+1}
    \arrow{r}{g_{\mu+1}} \arrow{d} & M_{\mu+1} \arrow{r} \arrow{d}& 0\\
    0 \arrow{r} & \frac{M_{\mu+1}}{M_\mu} \arrow{r}{\tilde f_{\mu}} &
    \frac{G_{\mu+1}}{G_\mu} \arrow{r}{\tilde g_\mu} &
    \frac{M_{\mu+1}}{M_{\mu}} \arrow{r} & 0
  \end{tikzcd}
\end{displaymath}
Since there exists a commutative diagram
\begin{displaymath}
  \begin{tikzcd}
    \bigoplus_{i \in J_\mu}\frac{C_i+G_\mu}{G_\mu} \arrow{r}{\hat
      f_\mu \hat g_\mu} \arrow{d}{\cong} & \hat f_\mu \hat g_\mu
    \left(\bigoplus_{i \in
        J_\mu}\frac{C_i+G_\mu}{G_\mu}\right) \arrow{d}{\cong}\\
    \frac{G_{\mu+1}}{G_\mu} \arrow{r}{\tilde g_\mu} & \frac{M_{\mu+1}}{M_\mu}\\
  \end{tikzcd}
\end{displaymath}
we conclude that $\tilde g_\mu$ is split. Then, by Lemma
\ref{l:lemma1}, $f_{\mu+1}$ is split too and $M_{\mu+1}$ is a direct
summand of $G_{\mu+1}$. Moreover, $\Img \tilde f_\mu$ is a direct
summand of $\frac{G_{\mu+1}}{G_\mu}$. But this image is precisely
$\frac{M_{\mu+1}+G_\mu}{G_\mu}$. This concludes the construction.
\end{proof}

We apply the previous result to get that some pure periodic modules
are trivial.

\begin{cor}\label{c:SplittingSequence}
  Let $\mathcal P$ be a class of finitely presented modules. Then each
  pure $\Add (\mathcal P)$-periodic module belongs to
  $\Add(\mathcal P)$. In particular we get the following:
  \begin{enumerate}
  \item \cite[Theorem 1.3]{Sim02} Each pure $\PProj$-periodic module
    is pure projective.
  \item \cite[Theorem 2.5]{BG} Each pure $\Proj$-periodic module is
    projective (equivalently, each flat $\Proj$-periodic module is projective).
  \end{enumerate}
\end{cor}

\begin{proof} Let $M$ be a pure $\Add(\mathcal P)$-periodic
  module. Then there exists a pure exact sequence
  \begin{displaymath}
    \begin{tikzcd}
      0 \arrow{r} & M \arrow{r} & P \arrow{r}{g} & M \arrow{r} & 0
    \end{tikzcd}
  \end{displaymath}
  with $P \in \Add (\mathcal P)$. By Lemma \ref{l:PropertiesOfG} (4),
  $M \in \lse (\mathcal P)$, so that there exists a locally split
  epimorphism $f\mathcolon Q \rightarrow M$ with $Q \in \Add(\mathcal P)$. Since
  $Q$ is pure projective by \cite[33.6]{Wis}, there exists
  $h\mathcolon Q \rightarrow P$ such that $gh=f$. Applying Lemma
  \ref{l:LocallySplitEpi} to the commutative diagram
  \begin{displaymath}
    \begin{tikzcd}
      Q \arrow[equal]{r} \arrow{d}{h} & Q \arrow{d}{f}\\
      P \arrow{r}{g} & M\\
    \end{tikzcd}
  \end{displaymath}
  we conclude that $g$ is locally split. By Theorem
  \ref{t:SplittingSequence} $g$ is actually split and
  $M \in \Add (\mathcal P)$.

\noindent
Now, to get (1) we simply note that $\PProj$ is equal to
$\Add(\mathcal P)$ for the class $\mathcal P$ of all finitely
presented modules. Finally to get (2), if $M$ is a pure
$\Proj$-periodic module then it is pure projective by (1). Since it is
flat too, we conclude that $M$ is projective.
\end{proof}

It is easy to see that the class of projective modules is not closed
under periodic modules. Actually, there exist $\Proj$-periodic modules
that are not pure projective.

\begin{example}\label{e:PeriodicNonProjective}
  Let $R$ be a QF ring which is not right pure semisimple. Then we can
  apply Corollary \ref{p:ConstructionPeriodic} with $\mathcal C$ the
  class $\Proj$ and $\mathcal D$ the class $\PProj$ to construct a
  $\Proj$-periodic module $M$ which is not pure projective. In
  particular, $M$ is not trivial. Since $\Flat=\Proj$ in this case,
  this also gives us an example of a $\Flat$-periodic module which is
  not trivial.
\end{example}

Another consequence of Corollary \ref{c:SplittingSequence} is that a
flat and strongly Gorenstein projective module is projective.

\begin{cor}
  Any flat and strongly Gorenstein projective module is projective.
\end{cor}

A natural question arises:
\begin{ques}\label{Q:1} Is every flat Gorenstein projective module
  projective?
\end{ques}

Using Corollary \ref{c:SplittingSequence} we can give a partial answer
to this question.

\begin{prop} \label{p:FlatGorensteinProj} Suppose that
  $\FPD_{\Flat}(R) < \infty$. If a cycle in an acyclic complex of projective modules is flat, then all cycles are projective modules. 
  In particular, every flat Gorenstein projective
  module is projective.
\end{prop}

\begin{proof}
Let $P$ be an acyclic complex of projectives and suppose that $M=
  Z_0(P)$ is flat. By \cite[36.6]{Wis}, $Z_n(P)$ is flat for each $n >
  0$.
  Now let $d=\FPD_{\Flat}(R)$ and $n < 0$. Since $Z_{n-d-1}(P)$ has
  finite flat dimension, it has flat dimension less than or equal to
  $d$. As $Z_n(P)$ is a syzygy of $M$, it has to be flat. The conclusion
  is that $Z_n(P)$ is flat for each $n \in \mathbb Z$.
Now each $Z_n(P)$ is a direct summand of
$\bigoplus_{n \in \mathbb Z}Z_n(P)$, which is a flat $\Proj$-periodic
module. By Corollary \ref{c:SplittingSequence}(2), the module
$\bigoplus_{n \in \mathbb Z}Z_n(P)$ is projective, and so is $Z_n(P)$, $n\in \mathbb{Z}$.
\end{proof}

\section{Periodic modules with respect to hereditary cotorsion pairs}
% \section{Preliminaries}

In this section we study the dual notion of flat $\Proj$-periodic
modules: $\Inj$-periodic absolutely pure modules. In order to do this,
we consider periodic modules with respect to the right class of a
hereditary cotorsion pair. Let us start with the following two lemmas,
which are useful for computing the $\Ext$ functors with periodic
modules.

\begin{lem}\label{L:Ext-1} Let $(\A, \B)$ be a hereditary 
  cotorsion pair in $\Modr R$ and let $M$ be a $\B$-periodic
  module. Then, for every module $L\in \A$ and non zero natrual nuber
  $n$, $\Ext^n_R(L, M)\cong\Ext^1_R(L, M)$.
\end{lem}
\begin{proof} Let $L\in \A$. Since $M$ is $\B$-periodic, there exists
  an exact sequence $0\to M\to B\to M\to 0$ with $B\in \B$. The usual
  long exact sequence of cohomology attained to this short exact
  sequence, gives us an exact sequence
$$\cdots \to \textrm{Ext}^n_R(L,B)\to \textrm{Ext}^n_R(L,M) \to \textrm{Ext}^{n+1}_R(L,M)\to \textrm{Ext}^{n+1}_R(L,B)\to\cdots. $$
Since $(\A,\B)$ is hereditary we have that $\textrm{Ext}^i_R(L,B)=0$,
for every $i\geq 1$. Therefore, it follows that
$\textrm{Ext}^n_R(L,M) \cong \textrm{Ext}^{n+1}_R(L,M)$, for every
$n\geq 1$. So we get our claim.
\end{proof}
Symmetrically we have:
\begin{lem}\label{L:Ext-1-A} Let $(\A, \B)$ be a hereditary
  cotorsion pair in $\Modr R$ and let $M$ be an $\A$-periodic
  module. Then, for every module $T\in \B$ and non zero natural number
  $n$, $\Ext^n_R(M, T)\cong\Ext^1_R(M, T)$.
\end{lem}
%
% \begin{proof} Let $X\in \A$.  If $\Ext^1_R(X, M)=0$ then by
%   dimension shifting $\Ext^i_R(X, M)=0$ for every $i\geq 1$, since
%   $\Ext^j(X, B)=0$ for every $j\geq 1$.  Assume that the minimum
%   natural number such that $\Ext^n_R(X, M)=0$ is greater than
%   $1$. By dimension shifting
%   $\Ext^{n-1}_R(X, M)\cong \Ext^n_R(X, M)$, a contradiction.
% \end{proof}

We shall use the following relative version of the 2-out-of-3 property
for a class of modules.

\begin{defn}\label{D:2-3} Let $\C$, $\D$ be two classes of modules. We
  say that $\D$ has the $2$-out-of-$3$ property with respect to $\C$
  if the following holds: for every exact sequence
  $0\to C_1\to C_2\to C_3\to 0$ in $\C$, if two of the $C_i$'s are in
  $\D$ then the the third term is in $\D$ too.\end{defn}

Now we prove that the left orthogonal of a periodic module with
respect to the right class of a hereditary cotorsion pair has the
relative 2-out-of-3 property.

\begin{lem}\label{L:2-out-3} Let $(\A, \B)$ be a hereditary  cotorsion pair in $\Modr R$ and let $M$ be a $\B$-periodic module.
  Then $^\perp M$ has the $2$-out-of-$3$ property with respect to
  $\A$.
\end{lem}
\begin{proof} Let $0\to X\to Y\to Z\to 0$ be an exact sequence in
  $\mathcal A$. We have the long exact sequence of cohomology
  \begin{align*}
    0 & \to \textrm{Hom}_R(Z,M)\to \textrm{Hom}_R(Y,M)\to \textrm{Hom}_R(X,M)\to \\
      &\to \textrm{Ext}^1_R(Z,M)\to \textrm{Ext}^1_R(Y,M)\to \textrm{Ext}^1_R(X,M)\to \\
      & \to \textrm{Ext}^2_R(Z,M)\to \textrm{Ext}^2_R(Y,M)\to \textrm{Ext}^2_R(X,M)\to\cdots\:.
  \end{align*}
  Now, if $X,Z\in ^\perp\!\! M$ the exact sequence
  $$\textrm{Ext}^1_R(Z,M)\to \textrm{Ext}^1_R(Y,M)\to
  \textrm{Ext}^1_R(X,M) $$
  gives us that $Y\in ^\perp\!\! M$.  If $X,Y\in ^\perp\!\! M$ the
  sequence
  $$\textrm{Ext}^1_R(X,M)\to
  \textrm{Ext}^2_R(Z,M)\to\textrm{Ext}^2_R(Y,M)$$
  together with Lemma \ref{L:Ext-1}, gives that $Z\in ^\perp\!\! M$.
  Finally, if $Y,Z\in ^\perp\!\! M$, the sequence
  $$\textrm{Ext}^1_R(Y,M)\to
  \textrm{Ext}^1_R(X,M)\to\textrm{Ext}^2_R(Z,M)$$
  together with Lemma \ref{L:Ext-1}, gives that $X\in ^\perp\!\! M$.
\end{proof}

We have also the symmetric statement.
\begin{lem}\label{L:2-out-3-A} Let $(\A, \B)$ be a hereditary
  cotorsion pair in $\Modr R$ and let $M$ be an $\A$-periodic module.
  Then $M^\perp $ has the $2$-out-of-$3$ property with respect to
  $\B$.
\end{lem}
We will prove now that under the hypothesis of Lemma \ref{L:2-out-3}
and assuming that $\mathcal{A}$ is also closed under pure epimorphic
images, the class $^\perp\! M\cap \A$ is closed under direct
limits. We start by showing that $^\perp\! M\cap \A$ is closed under
well ordered direct unions of pure submodules.
\begin{lem}\label{L:direct-unions}

  Let $(\A, \B)$ be a hereditary cotorsion pair in $\Modr R$ and let
  $M$ be a $\B$-periodic module. Assume that $\A$ is closed under pure
  epimorphic images. Let $K$ be the direct union of a chain
  $\{K_{\alpha}\mid \alpha< \lambda\}$ of pure submodules with
  $K_{\alpha}\in\!\! ^\perp M\cap \A $, for each
  $\alpha<\lambda$. Then $K\in \!\!^\perp M\cap \A $.
\end{lem}

\begin{proof}
  Consider the \emph{continuous} chain
  $\{L_{\alpha}\mid \alpha<\lambda\}$ of submodules of $K$ given by:
  \begin{itemize}
  \item $L_{\alpha}:=K_{\alpha}$ if $\alpha$ is succesor.
  \item $L_{\alpha}:=\bigcup_{\gamma<\alpha}L_{\gamma}$ if $\alpha$ is
    a limit ordinal.
  \end{itemize}
  Since $\mathcal A$ is closed under direct sums and pure epimorphic
  images, $\A$ is closed under direct limits by
  \cite[33.9]{Wis}. Therefore, $L_{\alpha}\in \mathcal A$, for each
  $\alpha<\lambda$. It is also clear that $K$ is the direct union of
  the continuous chain $\{L_{\alpha}\mid \alpha<\lambda\}$. Since the
  class of pure submodules of a given module is closed under direct
  unions \cite[33.8]{Wis}, we follow that $L_{\alpha}$ is a pure
  submodule of $K$, for each $\alpha<\lambda$.  Let us call
  $L_{\lambda}=K$. We prove now by induction that
  $L_{\alpha}\in \!\!^\perp M$, for every $\alpha\leq\lambda$. If
  $\alpha$ is succesor, $K_{\alpha}=L_\alpha$ which belongs to
  $\!\!^\perp M$ by hypothesis. Assume that $\alpha$ is a limit
  ordinal. For each $\gamma<\alpha$ we have the exact sequence
  \begin{equation}\label{eq:1}
    0\to L_{\gamma}\to L_{\gamma+1}\to \frac{L_{\gamma+1}}{L_{\gamma}}\to 0
  \end{equation}
  in which $L_{\gamma}$ is a pure submodule of $K$ so that, by
  \cite[33.3]{Wis}, it is a pure submodule of $L_{\gamma+1}$. Since
  $\mathcal A$ is closed under pure epimorphic images, we infer that
  $\frac{L_{\gamma+1}}{L_{\gamma}}\in \mathcal A$. By our induction
  hypothesis $L_{\gamma}$ and $L_{\gamma+1}$ belong to $ ^\perp M $.
  Thus, the exact sequence (\ref{eq:1}) has all its terms in
  $\mathcal A$ and therefore by Lemma \ref{L:2-out-3}, the quotient
  module $\frac{L_{\gamma+1}}{L_{\gamma}}$ belongs to $ ^\perp M$.
  Finally by Eklof Lemma (\cite[Lemma 6.2]{GT12}), we conclude that
  $L_{\alpha}$ lies in $ ^\perp M$.
\end{proof}

\begin{thm}\label{T:direct-limits}  Let $(\A, \B)$ be a hereditary  cotorsion pair in $\Modr R$ and let $M$ be a $\B$-periodic module. Assume that $\A$ is closed under pure epimorphic images.
  If $\{X_i; f_{ji}\mid i\leq j \in I\}$ is a direct system of modules
  in $^\perp M\cap \A $, then $\varinjlim\limits_{i\in I}X_i$ is in
  $^\perp M.$
\end{thm}
\begin{proof} By \cite[Corollary 1.7]{AR} we can assume that $I$ is an ordinal $\lambda$ so that the direct system is a $\lambda$-sequence of the form $\{X_{\alpha}; f_{\beta\alpha}\mid \alpha\leq\beta<\lambda\}.$ \\
  If $\lambda=\omega$, then the well known presentation of a countable
  direct limit (e.g. \cite[Lemma 2.12]{GT12}) gives a short exact
  sequence
 $$0\to \bigoplus_{\alpha<\omega}X_{\alpha}\to \bigoplus_{\alpha<\omega}X_{\alpha}\to \varinjlim\limits_{\alpha<\omega} X_{\alpha}\to 0.$$ This is an exact sequence with all its terms in $\mathcal A$ and where the first two terms belong to $ ^\perp M$. Hence by Lemma \ref{L:2-out-3}, we get that $\varinjlim\limits_{\alpha<\omega} X_{\alpha}\in \!\!^\perp M$.\\
 Now for an arbitrary limit ordinal $\lambda$, let
 \begin{equation}\label{eq:2}
   0\to K\to \bigoplus_{\alpha<\lambda}X_{\alpha}\to \varinjlim\limits_{\alpha<\lambda} X_{\alpha}\to 0
 \end{equation}
 be the canonical exact sequence associated to the direct limit. Since $(\mathcal A,\mathcal B)$ is hereditary, we get that $K\in \mathcal A$. Then, the exact sequence (\ref{eq:2}) has all its terms in $\mathcal A$, and clearly $\bigoplus_{\alpha<\lambda}X_{\alpha}\in \!\!^{\perp}M$.  So in view of Lemma \ref{L:2-out-3}, to get our claim we only need to show that $K\in \!\!^{\perp} M$. We use Lemma \ref{L:direct-unions} to prove this.\\
 As in the proof of \cite[Lemma 2.1]{GPGA} we have that $K$ is the
 direct union of a chain $\{K_{\alpha}\mid \alpha<\lambda\}$, where
 each $K_{\alpha}$ is a direct summand of
 $\bigoplus_{\alpha<\lambda}X_{\alpha}$. We need to check that the
 system $\{K_{\alpha}\mid \alpha<\lambda\}$ fulfills the requirements
 of Lemma \ref{L:direct-unions}, i.e.
 \begin{itemize}
 \item $K_{\alpha}\in \!\!^\perp M\cap \A$, for each $\alpha<\lambda$.
 \item $K_{\alpha}$ is pure in $K$, for each $\alpha<\lambda$.
 \end{itemize}
 Since both classes $\mathcal A$ and $ ^{\perp} M$ are closed under
 direct summands, we get that $K_{\alpha}\in \!\!^\perp M\cap \A $,
 for each $\alpha<\lambda$. For the second condition, note that
 $K_\alpha$ actually is a direct summand of $K$.
\end{proof}

We illustrate some consequences of the previous result.

\begin{prop}\label{P:Fp-inj} The following hold true:
  \begin{enumerate}
  \item An $\Inj$-periodic absolutely pure module is injective.
  \item Let $(\A, \B)$ be a hereditary cotorsion pair in $\Modr
    R$.
    Assume that $\A$ is closed under pure epimorphic images. Then
    every $\B$-periodic module is cotorsion. In particular, every
    $\Cot$-periodic module is trivial, i.e. cotorsion (and so every
    $\Inj$-periodic module is cotorsion).
  \item Assume that each finitely generated right ideal has finite
    flat dimension. Then each $\Inj$-periodic module is trivial.
  \end{enumerate}
\end{prop}
\begin{proof}
  (1) Let $M$ be an $\Inj$-periodic absolutely pure module. Then
  $\Ext^1_R(F, M)=0$ for every finitely presented module.  By
  Theorem~\ref{T:direct-limits}, $\Ext^1_R(X, M)=0$ for every module
  $X$, hence $M$ is injective.

\noindent
(2) Let $M$ be a $\B$-periodic module. Since $\Proj\subseteq \A$,
Theorem~\ref{T:direct-limits} implies that
$\Flat\subseteq ^\perp\!\! M$, hence $M$ is cotorsion.

\noindent
(3) Let $M$ be an $\Inj$-periodic module. We only have to show that
$\Ext^1_R\left(\frac{R}{I},M\right)=0$ for each right ideal $I$ of
$R$. Let $I$ be a right ideal and write
$I=\bigcup_{\gamma \in \Gamma}I_\gamma$ as a direct union of finitely
generated right ideals. By (2) $M$ is cotorsion so, for each
$\gamma \in \Gamma$, there exists a non zero natural number $n_\gamma$
such that $\Ext^{n_\gamma}_R(I_\gamma,M)=0$. By Lemma \ref{L:Ext-1}
$I_\gamma \in {^\perp}M$, and by Theorem \ref{T:direct-limits} we get
that actually $I$ belongs to ${^\perp}M$. But then, again by Lemma
\ref{L:Ext-1}, we get that
$\Ext^2_R\left(\frac{R}{I},M\right)\cong
\Ext^1_R\left(\frac{R}{I},M\right) = 0$.
\end{proof}

The same argument used in Example \ref{e:PeriodicNonProjective} can be
used to see that there exist $\Inj$-periodic modules which are not
pure injective.

\begin{example}
  Let $R$ be a QF ring which is not right pure semisimple. We can
  apply Corollary \ref{p:ConstructionPeriodic} with $\mathcal C$ the
  class $\Inj$ and $\mathcal D$ the class $\PInj$ to construct an
  $\Inj$-periodic module which is not pure injective. In particular,
  $M$ is a non-trivial $\Inj$-periodic module.
\end{example}

\begin{rem}
  Note that $\Cot$ does not satisfy the hyphotesis of Corollary
  \ref{p:ConstructionPeriodic}, so that we cannot use that result to
  construct a non-trivial $\Cot$-periodic module. This is because if
  $\Cot$ is generating, then $R_R$ is $\Sigma$-cotorsion (that is,
  $R^{(I)}$ is a cotorsion right module for each set $I$) and, by
  \cite[Corollary 10]{GH}, $R$ is right perfect. This means that
  $\Cot = \Modr R$. Here we give a short proof that right
  $\Sigma$-cotorsion rings are right perfect.
\end{rem}
\begin{thm}\label{t:sigmacot}
  If $R_R$ is $\Sigma$-cotorsion, then $R$ is right perfect.
\end{thm}

\begin{proof}
  Let $a_1R \geq a_1a_2 R \geq \cdots$ be a descending chain of
  principal right ideals for some sequence $a_1, a_2, \cdots$ of
  elements of $R$ and denote by $G$ and $F$ the modules constructed in
  \cite[Lemma 1.1]{Bass}. Then $F$ and $G$ are free and $G$ is a
  submodule of $F$. Moreover, it is proved in \cite[Lemma 1.1]{Bass}
  that $\frac{F}{G}$ is flat. Now, using that $G$ is cotorsion, as
  $R_R$ is $\Sigma$-cotorsion, we get that the short exact sequence
  $$0\longrightarrow G\longrightarrow F\longrightarrow \frac{F}{G}\longrightarrow 0$$
  splits. By \cite[Lemma 1.3]{Bass}, the sequence
  $a_1R \geq a_1a_2R \geq \cdots$ terminates. This implies that $R$ is
  right perfect.
\end{proof}

Now, regarding Proposition \ref{P:Fp-inj}(3), we show an example of a
ring with infinite weak global dimension but such that each finitely
generated right ideal has finite flat dimension.

\begin{example}
  Let $k$ be a field and let $k[x_1,\ldots,x_n]$ be the polynomial
  ring in $n$ variables over $k$. The ring
  $R=\varinjlim\limits_{n<\omega}k[x_1,\ldots,x_n]$ has infinite weak
  global dimension but each finitely generated ideal of $R$ has finite
  projective dimension (see Glaz, \cite[p. 202]{Glaz89}).
\end{example}

We can apply Proposition \ref{P:Fp-inj} to Gorenstein injective modules.
\begin{cor}\label{c:GorInj}
  \begin{enumerate}
  \item Any absolutely pure strongly Gorenstein injective module is
    injective.

  \item Assume that each finitely generated right ideal has finite
    flat dimension. Then each Gorenstein injective module is
    injective. In particular, if a ring has finite weak global
    dimension, then each Gorenstein injective module is injective.
  \item Assume that each finitely generated right ideal has finite
    injective dimension. Then each Gorenstein injective module is
    injective.
  \end{enumerate}
\end{cor}
\begin{proof} (1) and (2) follow directly from Proposition
  \ref{P:Fp-inj}.

  (3) Let $M$ be a strongly Gorenstein injective module. Since $M$ is
  in particular Gorenstein injective, $^\perp M$ contains the modules
  of finite injective dimension. Then the conclusion follows arguing
  as in the proof of Proposition~\ref{P:Fp-inj}~(3).
\end{proof}

We ask the following:
\begin{ques}\label{Q:2} Is every absolutely pure Gorenstein injective
  module injective?
\end{ques}

As in the case of flat Gorenstein projective modules, we can give a
partial answer to this question in the following result. The proof is
similar to the proof of Proposition \ref{p:FlatGorensteinProj}, but
using Corollary \ref{P:Fp-inj}(1).

\begin{prop}
  Suppose that $R$ is left coherent and $\FID_{\Abs}(R) < \infty$. If
  a cycle in an acyclic complex of injective modules is absolutely
  pure, then all cycles are injective modules.  In particular, every
  absolutely pure Gorenstein injective module is injective.
\end{prop}

\section{Acyclic complexes and $\C$-periodic modules
}\label{S:acyclic}

This section is devoted to exploit the power of periodic modules in
shortening and simplifying recent proofs of some meaningful results in
homotopy categories.

We apply the previous results to classes of complexes of
$R$-modules. The following result has the rank of Theorem, because of
its relevant statements and its subsequent consequences. But the proof
is an easy and immediate application of propositions \ref{P:cycles}
and \ref{P:Fp-inj}.
\begin{thm}\label{T:cot-Fp}
  \begin{enumerate}
  \item Every acyclic complex of injective modules with absolutely pure cycles
    is contractible,
    i.e. $\dw\Inj\cap\widetilde{\Abs}=\widetilde{\Inj}.$
  \item Every acyclic complex of cotorsion modules has cotorsion
    cycles, that is, $\ex\Cot$ $=\widetilde{\Cot}$. In particular,
    every acyclic complex of flat cotorsion modules with flat cycles
    is contractible, i.e.
    $\widetilde{\Flat}\cap \dw\Cot=\widetilde{\Flat\cap\Cot}$.
  \item Every pure acyclic complex of pure projective modules is
    contractible. As a consequence, if there exists a pure exact
    sequence
$$
0\to M\stackrel{f}{\to} P_n\to P_{n-1}\to \cdots\to P_1\to
P_0\stackrel{g}{\to} M\to 0
$$
in $\Modr R$, where the modules $P_0,\ldots, P_n$ are pure projective,
then $M$ is pure projective.
\end{enumerate}
\end{thm}
\begin{proof}
  The statements (1) and (2) are direct consequences of propositions
  \ref{P:cycles} and \ref{P:Fp-inj}. The assertion $(3)$ follows from
  Corollary \ref{c:SplittingSequence}(1) and (the pure version of)
  Proposition \ref{P:cycles}. Finally, the second part in statement
  (3) follows because from that pure exact sequence, we get the
  following pure acyclic complex with pure projective
  components
  $$\cdots\to P_0\stackrel{fg}{\to} P_n\to \cdots \to
  P_0\stackrel{fg}{\to} P_n\to\cdots$$
\end{proof}
\begin{rem}
  The assertions (1) and (3) were already proved by
  {\v{S}}{\v{t}}ov{\'{\i}}{\v{c}}ek \cite[Corollary
  5.5]{Stopurity} (see also Emmanouil \cite[Corollary 3.7]{Emm16} for (3)).
  The consequence in statement (3) was firstly shown
  by Simson in \cite[Theorem 1.3]{Sim02}. We just want to emphasize
  how these statements easily follow from the corresponding properties
  of periodic modules.  What seems to be unknown is the remarkable
  statement (2). Notice that, as a consequence of (2), we get that
  every acyclic complex of injectives has cotorsion cycles. This was also
  shown by {\v{S}}{\v{t}}ov{\'{\i}}{\v{c}}ek \cite[Corollary
  5.9]{Stopurity}.
\end{rem}

\begin{thm}\label{T:dgcotdwcot}
  Let $C$ be a complex of cotorsion modules. Then every chain map
  $f\mathcolon F\to C$, where $F\in\widetilde\Flat$ is null-homotopic.
  That is, the classes $\dg\Cot$ and $\dw\Cot$ coincide.
\end{thm}
\begin{proof}
  Let us consider the complete hereditary cotorsion pair
  $(\tilclass\Flat,\dg\Cot)$ in $\Ch(R)$ (\cite[Corollary
  4.10]{G1}). Then, there is an exact sequence
  $$0\to C\to D\to G\to 0,$$ with $D\in \dg\Cot$,
  $G\in \tilclass \Flat$ and the sequence splits on each degree. Since
  $D_n$ is cotorsion, we follow that $G_n$ is cotorsion for each
  $n \in \mathbb Z$. Therefore $G\in \tilclass \Flat\cap
  \dw\Cot$.
  Now, by Theorem \ref{T:cot-Fp}(2) we get that $G$ is
  contractible. Therefore $C$ and $D$ are homotopically equivalent,
  and so $C\in \dg\Cot$.
\end{proof}

\subsection{Application to finitely accessible additive categories} 
Throughout this section $\G$ will denote a {\it finitely accessible}
additive category. That is, $\G$ has all direct limits, the class of
finitely presented objects is skeletally small and every object in
$\G$ is a direct limit of finitely presented objects. A well-known
Representation Theorem (see \cite[Corollary 2.1.9]{MaPa},
\cite[Theorem 1.4]{CB} and \cite[Theorem 2.26]{AR} and the remark that
follow) states the following

\begin{thm}
  Every finitely accessible additive category $\mathcal G$ is
  equivalent to the full subcategory $\mathrm{Flat}(A)$ of the
  category $\operatorname{Mod-}\!\! A$ of unitary right $A$-modules
  consisting of flat right $A$-modules where $A$ is the functor ring
  of $\mathcal G$ (that is, a ring with enough idempotents). This
  equivalence gives a 1-1 correspondence between pure exact sequences
  in $\mathcal G$ and exact sequences in $\mathrm{Flat}(A)$.
\end{thm}
In other words, $\mathcal{G}$ with its pure exact structure $\E$ is
equivalent to $\mathrm{Flat}(A)$ with its canonical exact structure
inherited from $\Modr A$. In particular, the equivalence takes
injective objects in $(\mathcal{G};\mathcal E)$ (i.e. pure injectives)
to injective objects in $\mathrm{Flat}(A)$ (cotorsion flat modules).
Thus, from Proposition \ref{P:Fp-inj}(2), we immediately get the
following.
\begin{cor}\label{c:pInjperiodic.trivial}
  Every pure $\PInj$-periodic object of $\G$ is trivial (i.e. pure
  injective).
\end{cor}
\begin{proof}
  Let $M$ be a pure $\PInj$-periodic module and
  $0\to M\to E\to M\to 0$ a pure exact sequence with $E$ pure
  injective. By using the Representation Theorem, we get an exact
  sequence in $\Modr A$,
  $0\to \overline{M}\to \overline{E}\to \overline{M}\to 0$, with
  $\overline{M}$ flat and $\overline{E}$ flat cotorsion. But then, by Theorem \ref{T:cot-Fp}(2) (whose argument is still valid for a category of unital modules over a ring with enough idempotents) the sequence
  splits, so $\overline{M}$ is flat cotorsion and therefore $M$ is
  pure injective.
\end{proof}

The equivalence between $(\G;\E)$ and $\Flat(A)$ takes pure acyclic
complexes in $\Ch({\mathcal G})$ (i.e. concatenation of conflations in
$(\mathcal G;\E)$) to acyclic complexes in $\Ch({A})$ with flat
cycles. We will denote by $\Ch(\G)_{{\rm dw}\textrm{-}{\rm pur}}$ the exact
category of unbounded chain complexes $\Ch(\G)$ with the degreewise
pure exact structure.

The following result ({\v{S}}{\v{t}}ov{\'{\i}}{\v{c}}ek \cite[Theorem
5.4]{Stopurity}) can be also easily proved by using the Representation
Theorem for finitely accessible additive categories and Theorem
\ref{T:dgcotdwcot} (which still holds for unital modules over a ring with enough idempotents).
\begin{cor}\label{C:Stov.Cot.Pair}
  Let $\widetilde{\mathrm{PurAc}}$ be the class of pure acyclic chain 
  complexes in $\Ch(\G)$. The pair of clasess
  $(\widetilde{\mathrm{PurAc}},\dw\PInj)$ is a complete hereditary
  cotorsion pair in $\Ch(\G)_{{\rm dw}\textrm{-}{\rm pur}}$.
\end{cor}
\begin{proof}
  We have the complete hereditary cotorsion pair
  $(\tilclass\Flat,\dg\Cot)$ in $\Ch(A)$. Then we have the induced
  complete hereditary cotorsion pair
  $(\tilclass\Flat,\dg\Cot\cap \dw\Flat)$ in $\Ch(\Flat)$ (see for
  instance \cite[Corollary 7.5]{G7}). Now, by Theorem
  \ref{T:dgcotdwcot}, $\dg\Cot=\dw\Cot$. Therefore the previous
  cotorsion pair is $(\tilclass\Flat,\dw\Cot\cap \dw\Flat)$. Now we use
  the Representation Theorem to get the complete hereditary cotorsion
  pair $(\widetilde{\mathrm{PurAc}}, \dw\PInj)$ in
  $\Ch(\G)_{{\rm dw}\textrm{-}{\rm pur}}$.
\end{proof}
\begin{rem}
For any complex $E$ of pure injective objects in $\G$ and any complex $A$, we have that $\Ext^1_{\Ch(\G)_{{\rm dw}\textrm{-}{\rm pur}}}(A,E)=0$  if and only if every map from $A\to \Sigma E$ is null-homotopic. Then it follows from Corollary \ref{C:Stov.Cot.Pair} that a complex $A$ is pure acyclic if and only if any map $A\to I$ is null-homotopic, where $I$ is a complex of pure injective objects in $\G$.
\end{rem}
\bibliographystyle{alpha} \bibliography{references}
\end{document}